\documentclass[12pt]{amsart}

\usepackage{amsmath, amssymb}
\usepackage{enumitem}
\newtheorem{theorem}{Theorem}[section]
\newtheorem{lemma}[theorem]{Lemma}
\newtheorem{prop}[theorem]{Proposition}
\newtheorem{cor}[theorem]{Corollary}

\theoremstyle{definition}

\theoremstyle{remark}
\newtheorem{remark}[theorem]{Remark}

\numberwithin{equation}{section}

\newcommand{\rd}{{\mathbb R^d}}
\newcommand{\rr}{{\mathbb R}}
\newcommand{\integer}{{\mathbb Z}}
\newcommand{\N}{{\mathbb N}}

\newcommand{\QQd}{\mathbb{Q}(d)}
\renewcommand{\d}{\operatorname{d}}
\newcommand{\p}{\mathcal{P}}
\renewcommand{\P}{\mathbb{P}}
\newcommand{\lex}{<_{\text{lex}}}

\allowdisplaybreaks

\renewcommand{\P}{\mathbb{P}}

\newcommand{\E}{\mathbb{E}}

\renewcommand{\d}{\operatorname{d}}
\newenvironment{proofof}[1]{\vskip 2mm\noindent {\textit{Proof of {#1}~}}}
                    {\hfill $\square$ \vskip 2mm \noindent}
\usepackage{xcolor}

\usepackage{booktabs}

\usepackage{multirow}

\usepackage{tabularx}

\usepackage{graphicx}

\newcolumntype{C}[1]{>{\centering\arraybackslash}m{#1}}

\begin{document}
\sloppy
\title{The longest edge in discrete and continuous long-range percolation }

\author{Arnaud Rousselle}
\address{Arnaud Rousselle, Institut de Mathématiques de Bourgogne, UMR 5584, CNRS, Université de Bourgogne, F-21000 Dijon, France}
\email{arnaud.rousselle@u-bourgogne.fr}

\author{Ercan S\"onmez}
\address{Ercan S\"onmez, Faculty of Mathematics, Bochum, Germany}
\email{ercan.soenmez\@@{}rub.de}

\begin{abstract}
We consider the random connection model in which an edge between two Poisson points at distance $r$ is present with probability $g(r)$. We conduct an extreme value analysis on this model, namely by investigating the longest edge with at least one endpoint within some finite observation window, as the volume of this window tends to infinity. We show that the length of the latter, after normalizing by some appropriate centering and scaling sequences, asymptotically behaves like one of each of the three extreme value distributions, depending on choices of the probability $g(r)$. We prove our results by giving a formal construction of the model by means of a marked Poisson point process and a Poisson coupling argument adapted to this construction. In addition, we study a discrete variant of the model. We obtain parameter regimes with varying behavior in our findings and an unexpected singularity. 
\end{abstract}


\maketitle

\baselineskip=18pt
\sloppy

\section{Introduction}

Random graphs are used as models for networks, specifically capturing their complexity by describing local and probabilistic rules according to which vertices are present and connected to each other. Long-range percolation is a very popular random graph, studied extensively for more than several decades now. A classical example is long-range percolation on the lattice, in which the vertex set is $\integer^d$ and an edge between an arbitrary pair of vertices $x,y \in \integer^d$ is present with probability $p(x,y)$ depending only on their Euclidean distance, independently of all the other edges. We refer to this model as discrete long-range percolation in this paper. A continuous counterpart is given by the random connection model, with the crucial difference of the vertex set being derived from the canonical projection of a stationary Poisson point process. We call such a model continuous long-range percolation in this paper. 

Since its introduction in \cite{schulman} discrete long-range percolation has been the subject of a considerably large number of works, see \cite{Benber, berger1, Biskup, CGS, CS1, CS2, trapman} and references therein, just to mention a few. Significant amount of these works focuses specifically on graph distances and the behavior of the random walk. One of the first works studying continuous long-range percolation \cite{Penrose1} addresses its percolation behavior and the study of cluster sizes. Ever since this model has also received great attention, see \cite{burton, devroye, iyer, last2, last1, Penrose2, s2,s1, van} for example. Recent works focus on central limit theorems for several graph structures in continuous long-range percolation such as component counts. More recently, we also encounter properties of the random walk \cite{s1} and the graph distances \cite{s2} on this graph.

In this work we study a problem jointly for both models from another perspective and elucidate yet unrevealed properties long-range percolation exhibits. Given the construction of the models it is an eligible and natural question of interest how long an edge can possibly be. A precise mathematical formulation of this problem is given in Section 2. The nature of this question falls within the scope of extreme value theory, which is mainly concerned with max-stable random elements occurring as limits of normalized maximums. 

Most of the works on long-range percolation assume that the probability of connecting two vertices has a polynomial decay in their distances. A little number of early works (see \cite{hara, meester}) includes exponentially decaying probabilities. In this paper we consider three classes of functions for the probability of an edge to be present, depending on the Euclidean distance between the corresponding vertices. The first class  is the widely-studied class of functions with polynomial decay. The second class is the class of functions with exponential decay. Finally, the third class is a class of functions under which the lengths of the edges are bounded by some positive and finite constant (a finite right endpoint) and have a power law behavior at this constant, see Section 2 below.

The purpose of this paper is to determine the asymptotic behavior of the maximum length of edges with at least one endpoint within some finite observation window, as the volume of this window tends to infinity. We discover that the aforementioned maximum, after normalizing by some appropriate centering and scaling constants, asymptotically behaves like an extreme value distribution. More precisely, depending on the choice of the probability functions mentioned above, we recover the Fr\'echet, Gumbel and Weibull distribution in the limit.

In the case of polynomially decaying probabilities the edge connecting probability depends on an exponent $\alpha$, which is assumed to be larger than the underlying dimension $d$. In this setting, substantial works on graph-theoretical properties have revealed phase transitions with typical behavior according to the values of $\alpha$, see \cite{Biskup} for details. Surprisingly, in our investigations we encounter a phase transition with solely two phases and a singularity with unexpected behavior in case $\alpha=2d$, see Theorem \ref{thd2} below for details. Additionally, we show that this phenomenon does not occur if one considers directed edges, see Theorem \ref{prop:MaxEdgedLRPq}. We again remark that the vertex set of discrete long-range percolation is $\integer^d$, whereas the vertex set of the random connection model can be an arbitrary countable subset of $\rd$. Also, a mentionable difference between the discrete and continuous model is that obtaining a Weibull limit fails in the discrete case, see Remark \ref{rwd}. 


Our results for discrete and continuous long-range percolation can be considered to be analogous counterparts. Nevertheless, the results and methods presented here in obtaining them exhibit subtle, but crucial differences. In order to establish our Theorems \ref{thd1} and \ref{thd2} for the discrete model we make use of the law of small numbers, see \cite{ross}. Our Theorems \ref{th:MaxEdgeLRPc} and \ref{thcnew} for the continuous model are accomplished by employing a general result on Poisson approximation \cite[Theorem 3.1]{Penrose2}, which is related to Stein's method, see \cite{barbour, lindvall} for example. Such a result in the case that the underlying Poisson point process is finite has been used in \cite{Penrose2} to study central limit theorems for component counts in some random connection models derived from finite Poisson point processes, but is also available if the underlying Poisson point process has $\sigma$-finite but infinite intensity measure. The latter version of this Poisson approximation Theorem turns out to be of use for our purposes. More particularly, it enables us to obtain convergence of the sum of exceedances of normalized edge lengths towards a Poisson distribution in terms of the total variation and Wasserstein distance, additionally gives us an insight on the rates regarding the convergence speed. The application of the abstract Poisson convergence result requires a formal way of constructing the random connection model by means of a marked point process, which we give in the proof of our result in the continuous counterpart.

We remark that long edges have also been studied for the minimal spanning tree derived from (finite) Poisson point processes in \cite{penrose3} under a different framework. Since there is no additional randomness in the presence of the edges, there the construction of the random graph is completely determined by the position of vertices and the proof makes use of the Chen-Stein method \cite{arr}. {One of the important challenges in this paper is that the construction is not local and that there is an additional randomness in the construction of the edge set itself. This will, in particular, lead us to introduce suitable couplings in both discrete and continuous cases and to provide a construction of the continuous model from marked Poisson point processes (see Subsection \ref{ssec:construction}).}

We close the Introduction with a brief description of the structure of this paper. In Section 2, after clarifying the models under consideration, we rigorously formulate the problem to be investigated and we state the main assertions of this paper. Finally, we split the proofs for the discrete model and the continuous counterpart in Sections 3 and 4, respectively.

\section{Behavior of the longest edge}

\subsection{Undirected edges}
Consider two models of long-range percolation on the $d$-dimensional Euclidean space, $d \in \N$, defined as a graph $(V,E)$. For discrete long-range percolation, the vertex set is given by $V=V^{\text{dis}}=\integer^d$ while for continuous long-range percolation, it is given by $V=V^{\text{cont}}=\mathcal{P}$ where $\mathcal{P}=\{X_n: n \in \N\}$ is a realization of a homogeneous Poisson point process with unit intensity in $\rd$, see \cite{SW, last3} for example. Given the set of vertices $V$ we construct the set of undirected edges as follows. Let $g\colon \rd \to [0,1]$ be a measurable function satisfying $g(x) = g(-x)$ for all $x\in\rd$ and
\begin{equation} \label{g1}
 0< \int_\rd g(x) dx < \infty.
\end{equation}
For every $x,y \in V$ the (undirected) edge connecting $x$ and $y$ is present with probability $g(x-y)$, independently of all the other edges. We will use the notation $x\leftrightarrow y$ if there is an (undirected) edge $\{x,y\} \in E$. We refer to Section 4 for a formal construction in case the vertex set is derived from the homogeneous Poisson point process. In the sequel we denote by $\mathbb{P}$ the corresponding probability measure and by $\mathbb{E}$ the expectation with respect to this measure. We now fix $\|\cdot\|$ a norm on $\rd$. If $V=V^{\text{dis}}$ then we assume that $\|\cdot\|$ is the 1-norm. If $V=V^{\text{cont}}$ we assume that $\|\cdot\|$ is the classical Euclidean norm, that is the 2-norm. For $n\in\N$, let $B_n = [-n,n]^d$ denote an observation window. We then consider the length of the longest edge with at least one endpoint in $B_n$:
\begin{align}\label{m1}
 e_n^* &:= \max_{ \{x,y\} \in E_{B_n}} \|x-y\|, \qquad E_{B_n}:= \Big\{ \{x,y\} \in E: x \in B_n \Big\}.
\end{align}
Condition \eqref{g1} ensures that the random graph $(V,E)$ is almost surely locally finite. In particular this ensures that the maximum in \eqref{m1} is well-defined. We are mainly interested in the following types of functions satisfying \eqref{g1}. Our strategy can be adapted to slightly more general functions $g$, but we focus our attention to the following particular choices of functions.
\begin{itemize}
	\item[(F)] Fix a parameter $\alpha \in (d,\infty)$. We suppose that $g$ is a function satisfying
$$ \lim_{\|z\| \to \infty} \frac{g(z)}{\|z\|^{-\alpha}} = 1.$$
	\item[(G)] Fix parameters $\lambda, \alpha \in (0,\infty)$ and suppose that $g$ satisfies
$$\lim_{\|z\| \to \infty} \frac{g(z)}{\exp(-\lambda\|z\|^{\alpha})} =  1.$$
	\item[(W)] Fix parameters $M, \alpha \in (0,\infty)$ and let
 $$g(z)= M^{-\alpha}(M-{\|z\|})^{\alpha} \mathbf{1}_{\{ \|z\| \leq M \}}, \quad z \in \rd.$$
\end{itemize}
Recall that the cumulative distribution function of a Fr\'echet random variable with parameter $\beta\in (0,\infty)$ is given by $\Phi_\beta (r) = \exp(-r^{-\beta}) \mathbf{1}_{(0, \infty)} (r)$, $r \in \rr$, that the cumulative distribution function of a Gumbel random variable is given by $\Lambda(r) = \exp (-e^{-r})$, $r\in\rr$, and that the one of a Weibull random variable with parameter $\gamma \in (0, \infty)$ is given by $\Psi_\gamma (r) = \exp(-(-r)^\gamma)\mathbf{1}_{(-\infty, 0]}(r) + \mathbf{1}_{(0, \infty)} (r)$, $r \in\rr$.

Next we define a sum of exceedances, a natural setting for the study of extreme values and then we show that this sum has a Poisson limit. Now we make this discussion more precise. Given $r\in\rr$ and sequences $(b_n)_{n \in\N}$, $(c_n)_{n \in\N}$, let $r_n = c_nr + b_n$, $n\in\N$. The number of exceedances is defined as a random variable $W(n,r)$, $n\in\N$, $r\in\rr$, which will be given by
\begin{align}\label{def:NbExceed}
\begin{split}
W(n,r)&=  \sum_{x\in V \cap B_n} \sum_{y\in V \cap B^C_n} \mathbf{1}_{\{ \|x-y\| \mathbf{1}_{\{x\leftrightarrow y\}} >r_n \}} + \frac12 \sum_{x\in V \cap B_n} \sum_{y\in V \cap B_n} \mathbf{1}_{\{ \|x-y\| \mathbf{1}_{\{x\leftrightarrow y\}} >r_n \}}\\
& = \sum_{x\in V \cap B_n} \sum_{y\in V } \mathbf{1}_{\{ \|x-y\| \mathbf{1}_{\{x\leftrightarrow y\}} >r_n \}} - \frac12 \sum_{x\in V \cap B_n} \sum_{y\in V \cap B_n} \mathbf{1}_{\{ \|x-y\| \mathbf{1}_{\{x\leftrightarrow y\}} >r_n \}}.
\end{split}
\end{align}

Note that the correction term in \eqref{def:NbExceed} in this representation arises due to the fact that an exceedance with both endpoints in the observation window $B_n$ is counted twice in the number of vertices that are an endpoint of a long edge. A crucial observation is that the correction term disappears if $r_n \geq 2dn$ in the discrete case and $r_n \geq 2\sqrt{d}n$ in the continuous case, which will be the case in some of the regimes considered below.

{In our main results, we will provide choices of the sequences $(b_n)_{n \in\N}$ and $(c_n)_{n \in\N}$, according to the connection function $g$, under which the following convergence takes place:
\begin{equation}\label{eq:cvWn}
    W(n,r)  \xrightarrow[n\to \infty]{d} W (r),
\end{equation}
where $W(r)$ denotes a Poisson random variable with mean depending on $r$.
}

In the rest of the paper, we denote by $\omega_d$  the volume of the $d$-dimensional unit ball, with respect to the 1-norm in the discrete case and the 2-norm in the continuous case.

We begin with the discrete case. First, we state the following result, which only includes the connection function in case (F) for $\alpha \in (d,2d)$ and partly $\alpha =2d$. Recall that $\|\cdot\|$ is the 1-norm, which is supposed throughout whenever we consider the discrete model. Recall also that $W(n,r)$ is the number of exeedances and thus, $\{c_n^{-1}(e^*_n-b_n)\leq r\}=\{W(n,r)=0\}$. 

\begin{theorem}\label{thd1} Let $V=V^{\text{dis}}$. Assume that the function $g$ is given by \textnormal{(F)} with the parameter $\alpha \in(d,\infty)$, let ${c}_n= ((\alpha-d)2^{-d}d^{-1}\omega_d^{-1})^\frac{1}{d-\alpha}n^\frac{d}{\alpha-d}$ and $b_n = 0$, $n\in \N$.
\begin{enumerate}

\item \label{d11} If $\alpha \in(d,2d)$  then, for every $r\geq 0$, \eqref{eq:cvWn} holds with $\mathbb{E}[W(r)] = r^{d-\alpha}$. In particular,
$$ c_n^{-1} e_n^*  \xrightarrow[n\to \infty]{d} \Phi_{\alpha-d},$$
where $\Phi_{\alpha-d}$ has a Fr\'echet distribution with parameter $\alpha-d$.
\item\label{d12} If $\alpha =2d$ then, for every $r\geq 2d ((\alpha-d)2^{-d}d^{-1}\omega_d^{-1})^\frac{1}{\alpha-d}$, \eqref{eq:cvWn} holds with $\mathbb{E}[W(r)] = r^{d-\alpha}$.
\end{enumerate}
\end{theorem}

Theorem \ref{thd1} is accomplished by an application of the law of small numbers, see \cite[Theorem 4.6]{ross}. One may wonder why the assertion in case $\alpha =2d$ only holds if we assume $r\geq 2d ((\alpha-d)2^{-d}d^{-1}\omega_d^{-1})^\frac{1}{\alpha-d}$. For the other cases, namely for example if $\alpha = 2d$ and $r \in (0,2d ((\alpha-d)2^{-d}d^{-1}\omega_d^{-1})^\frac{1}{\alpha-d})$, or if $\alpha > 2d$ the correction term in \eqref{def:NbExceed} has a non-trivial contribution to the number of exceedances, resulting in the fact that the dependencies become too strong and the application of the law of small numbers becomes intractable. At this point it is an open and interesting question to find out whether we can complete the other cases as well. At least some progress can be made in dimension $d=1$ and we can definitely answer the question that in case $d=1$ and $\alpha=2$ the parameter of the Poisson limit, to our big surprise, changes if $r < 2d ((\alpha-d)2^{-d}d^{-1}\omega_d^{-1})^\frac{1}{\alpha-d}=\frac{1}{2}$. Indeed, our next main Theorem reads as follows.

\begin{theorem}\label{thd2} Let $V=V^{\text{dis}}$. Suppose that $d=1$. 
\begin{enumerate}
\item\label{d22} If the function $g$ is given by \textnormal{(F)} with parameter $\alpha=2$, let ${c}_n=4n$ and $b_n=0$, $n\in\N$. Then, for every {$r\geq \frac12$} \eqref{eq:cvWn} holds with $\mathbb{E}[W(r)] =r^{-1}$, and for every $r\in (0,\frac12)$, \eqref{eq:cvWn} holds with {$\mathbb{E}[W(r)] =\frac{1+2r}{2r} - \log (2r)$}. In particular,
$$ c_n^{-1} e_n^*  \xrightarrow[n\to \infty]{d} Z^*,$$
where $Z^*$ is a random variable with heavy-tailed distribution satisfying
\[\mathbb{P} (Z^* \leq r)  = \exp \left( -\frac{1+2r}{2r} + \log (2r) \right)\mathbf{1}_{(0,\frac12)}(r) +\exp \left( -r^{-1} \right)  \mathbf{1}_{[\frac12,\infty)}(r)\]
for every $r\in (0, \infty) $.
\item\label{d21} If the function $g$ is given by \textnormal{(F)} with the parameter {$\alpha>2$}, let ${c}_n= ((\alpha-1)/4)^\frac{1}{1-\alpha}n^\frac{1}{\alpha-1}$, $b_n = 0$, $n\in\N$. Then, for every $r\geq 0$, \eqref{eq:cvWn} holds with $\mathbb{E}[W(r)] = r^{1-\alpha}$. In particular,
$$ c_n^{-1} e_n^*  \xrightarrow[n\to \infty]{d} \Phi_{\alpha-1},$$
where $\Phi_{\alpha-1}$ has a Fr\'echet distribution with parameter $\alpha-1$.
\item\label{dthg2}  If the function $g$ is given by \textnormal{(G)} with parameters $\lambda,\alpha \in (0, \infty)$, define $K=K_{\alpha}= 4\Gamma(\frac{1}{\alpha})\alpha^{-1}\lambda^{-\frac{1}{\alpha}}$. Further define $$b_n=\lambda^{-\frac{1}{\alpha}}\left(\ln n+\left(\frac{1}{\alpha}-1\right)\ln\left(\ln n +\ln K\right)+\ln\left(K\Gamma\left(\frac{1}{\alpha}\right)^{-1}\right)\right)^\frac{1}{\alpha}$$ and $c_n =\alpha^{-1}\lambda^{-1}b_n^{1-\alpha}$ for every $n\in\N$. Then, for every $r\in\rr$, \eqref{eq:cvWn} holds with $\mathbb{E}[W] = e^{-r}$. In particular,
$$ c_n^{-1} (e_n^* - b_n)  \xrightarrow[n\to \infty]{d} \Lambda,$$
where $\Lambda$ has a Gumbel distribution. 
\end{enumerate}
\end{theorem}

{The biggest surprise is the fact that the correction term only has a substantial impact in the case $\alpha=2$ for the Fr\'echet case, but not if $\alpha>2$, which we also conjecture to hold for arbitrary dimension $d \geq 2$.}

We now turn to the continuous case. 
{\begin{theorem}\label{th:MaxEdgeLRPc} Let $V=V^{\text{cont}}$. Assume that the function $g$ is given by \textnormal{(F)} with the parameter $\alpha \in(d,\infty)$, let ${c}_n= ((\alpha-d)2^{-d}d^{-1}\omega_d^{-1})^\frac{1}{d-\alpha}n^\frac{d}{\alpha-d}$ and $b_n = 0$, $n\in \N$.
\begin{enumerate}
\item\label{cthf}
 Assume that $\alpha\in(d,2d)$. Then, for every $r\geq 0$, \eqref{eq:cvWn} holds with $\mathbb{E}[W(r)] = r^{d-\alpha}$. In particular,
$$ c_n^{-1} e_n^*  \xrightarrow[n\to \infty]{d} \Phi_{\alpha-d},$$
where $\Phi_{\alpha-d}$ has a Fr\'echet distribution with parameter $\alpha-d$.
\item\label{cthfnew1} If $\alpha =2d$ then, for every $r\geq 2\sqrt{d} ((\alpha-d)2^{-d}d^{-1}\omega_d^{-1})^\frac{1}{\alpha-d}$, \eqref{eq:cvWn} holds with $\mathbb{E}[W(r)] = r^{d-\alpha}$.
\end{enumerate}
\end{theorem}}

We notice that the content in Theorem \ref{th:MaxEdgeLRPc} mainly matches the one of Theorem \ref{thd1} in the discrete case. There is just one subtle difference regarding the case $\alpha=2d$, as in part \eqref{cthfnew1} we have a slightly different lower bound on $r$, which is due to the choice of different norms.

As in Theorem \ref{thd2} we continue the results in the continuous case by obtaining the same singularity for the case $\alpha=2d$.

\begin{theorem}\label{thcnew} Let $V=V^{\text{cont}}$. Suppose that $d=1$. 
\begin{enumerate}
\item\label{d22n} If the function $g$ is given by \textnormal{(F)} with parameter $\alpha=2$, let ${c}_n=4n$ and $b_n=0$, $n\in\N$. Then, for every {$r\geq \frac12$} \eqref{eq:cvWn} holds with $\mathbb{E}[W(r)] =r^{-1}$, and for every $r\in (0,\frac12)$, \eqref{eq:cvWn} holds with {$\mathbb{E}[W(r)] =\frac{1+2r}{2r} - \log (2r)$}. In particular,
$$ c_n^{-1} e_n^*  \xrightarrow[n\to \infty]{d} Z^*,$$
where $Z^*$ is a random variable with heavy-tailed distribution satisfying
\[\mathbb{P} (Z^* \leq r)  = \exp \left( -\frac{1+2r}{2r} + \log (2r) \right)\mathbf{1}_{(0,\frac12)}(r) +\exp \left( -r^{-1} \right)  \mathbf{1}_{[\frac12,\infty)}(r)\]
for every $r\in (0, \infty) $.
\item If the function $g$ is given by \textnormal{(F)} with the parameter {$\alpha>2$}, let ${c}_n= ((\alpha-1)/2)^\frac{1}{1-\alpha}n^\frac{1}{\alpha-1}$, $b_n = 0$, $n\in\N$. Then, for every $r\geq 0$, \eqref{eq:cvWn} holds with $\mathbb{E}[W(r)] = r^{1-\alpha}$. In particular,
$$ c_n^{-1} e_n^*  \xrightarrow[n\to \infty]{d} \Phi_{\alpha-1},$$
where $\Phi_{\alpha-1}$ has a Fr\'echet distribution with parameter $\alpha-1$.
\item\label{cthg2}
  If the function $g$ is given by \textnormal{(G)} with parameters $\lambda,\alpha \in (0, \infty)$ define $K=K_{ \alpha}= 4\Gamma(\frac{1}{\alpha})\alpha^{-1}\lambda^{-\frac{1}{\alpha}}$. Further define 
  \begin{equation} \label{gumbelbn}
      b_n=\lambda^{-\frac{1}{\alpha}}\left(\ln n+\left(\frac{1}{\alpha}-1\right)\ln\left(\ln n +\ln K\right)+\ln\left(K\Gamma\left(\frac{1}{\alpha}\right)^{-1}\right)\right)^\frac{1}{\alpha}
  \end{equation}
  and $c_n =\alpha^{-1}\lambda^{-1}b_n^{1-\alpha}$ for every $n\in\N$. Then, for every $r\in\rr$, \eqref{eq:cvWn} holds with $\mathbb{E}[W(r)] = e^{-r}$. In particular,
$$ c_n^{-1} (e_n^* - b_n)  \xrightarrow[n\to \infty]{d} \Lambda,$$
where $\Lambda$ has a (standard) Gumbel distribution.
\item\label{cthw}
 If the function $g$ is given by \textnormal{(W)} with parameters $M,\alpha \in (0, \infty)$, let ${c}_n = Kn^{-\frac{1}{\alpha+1}}$,  where $K=K_{  M,\alpha}=(M^\alpha(\alpha+1)2^{-2})^\frac{1}{\alpha+1}$, and let $b_n =M$, for every $n\in\N$. Then, for every $r\leq 0$, \eqref{eq:cvWn} holds with $\mathbb{E}[W(r)] = (-r)^{\alpha+1}$. In particular,
$$ c_n^{-1} (e_n^* - M)  \xrightarrow[n\to \infty]{d} \Psi_{\alpha+1},$$
where $\Psi_{\alpha+1}$ has a Weibull distribution with parameter $\alpha+1$.
\end{enumerate}
\end{theorem}


The Poisson limits in Theorem \ref{th:MaxEdgeLRPc} and \ref{thcnew} are achieved by proving an upper bound on both the total variation and Wasserstein distance between the sum of exceedances $W(n,r)$ and its Poisson limit $W(r)$. For example, in the Fr\'echet case we also see from the proof in Section 4 that both in the total variation and Wasserstein distance the convergence is of order $O(n^{-d})$, as $n\to \infty$. Again we note that there is just one subtle difference regarding the case $\alpha>2$ in Theorem \ref{thcnew}, as we have a slightly different constant in part (2) compared to the discrete case.


{A short summary of the contents in this subsection is provided in Table \ref{tab:disc}.}

\begin{table}[ht]
{
    \begin{center}
\begin{tabular}{c|c|c}
Case &  Value of $\mathbb{E}[W(r)]$ & Limit result\\
\hline
(F), $d \in \N$, $\alpha\in(d,2d)$ & $r^{d-\alpha}$ &   $c_n^{-1} e_n^*  \xrightarrow[n\to \infty]{d} \Phi_{\alpha-d}$\\
\hline
(F), $d =1$, $\alpha=2$ & $ \left\{
\begin{array}{ll}
\frac{1+2r}{2r} - \log (2r) , &  r\in (0,\frac12) \\
r^{-1}, & r \in [\frac12, \infty) \
\end{array}
\right. $ &   $c_n^{-1} e_n^*  \xrightarrow[n\to \infty]{d} Z^*$\\
\hline
(F), $d =1$, $\alpha\in(2, \infty)$ & $r^{1-\alpha}$ &   $c_n^{-1} e_n^*  \xrightarrow[n\to \infty]{d} \Phi_{\alpha-1}$\\
\hline
{(G)}, $d=1$ & $e^{-r}$ &  $ c_n^{-1} (e_n^* - b_n)  \xrightarrow[n\to \infty]{d} \Lambda$\\
\hline
(W), $d=1$, $V=V^{cont}$ & $(-r)^{\alpha+1}$ &   $c_n^{-1} (e_n^* - M)  \xrightarrow[n\to \infty]{d} \Psi_{\alpha+1}$\\
\end{tabular}
\end{center} 
\caption{Summary of the main results for undirected edges.
}
    \label{tab:disc}
}
\end{table}

\subsection{Directed edges}
We now investigate related oriented models, in which we consider directed edges of long-range percolation. Our findings seem worth formulating for us. Indeed, for such edges we obtain a complete picture, as we are able to handle arbitrary dimension. Moreover, we find that there is no singularity anymore in the critical case $\alpha=2d$, so that adding directions changes the behavior of the longest edge. Furthermore, as a by-product we get an additional insight into the behavior of undirected edges, see Corollary \ref{th:MaxEdgeLRPd1} below.

We consider the following oriented long-range percolation models. 
\medskip

\noindent\underline{Directed Long-Range Percolation (dLRP):} We call \emph{directed long-range percolation model} the random graph with vertex set $V$ in which an oriented edge $(u,v)$ is independently drawn from $u$ to $v$ with probability $g(v-u)$. We write $x \rightarrow y$ if there is an oriented edge from $x$ to $y$. Similarly, we define the maximum edge-length as
\begin{align*}
 e_n^*(\text{dLRP}) &:=\max_{x\in V\cap B_n} \max_{y\in V} \|x-y\| \mathbf{1}_{\{x \rightarrow y\}}.
\end{align*} Define the random variable $W^\rightarrow(n,r)$, $n\in \N$, as the sum of exceedances, which is now given by
\begin{align} \label{direxc}
    W^\rightarrow(n,r) = \sum_{x \in V\cap B_n} \sum_{y \in V}  \mathbf{1}_{\{x \rightarrow y, \|x-y\|>r_n \}}.
\end{align}

In contrary to \eqref{def:NbExceed}, there is no correction term anymore, which makes the study of exceedances easier. For example, in the discrete case, the random variables
\[e_{x_i}^*(\text{dLRP})=\max_{y\in \integer^d} \|x_i-y\| \mathbf{1}_{\{x_i \rightarrow y\}}\]
are independent by default so that the law of
\begin{align*}
 e_n^*(\text{dLRP}) &:=\max_{x_i\in B_n}e_{x_i}^*(\text{dLRP})
\end{align*}
can be determined explicitly according to $g$. {In this setting, the number of exceedances follows a binomial distribution whose parameters can be easily determined according to $n$, $d$ and $g$. This is a notable difference from undirected models in which long-range dependencies occur due to the non-local construction.} 

Let us observe that (discrete and continuous) long-range percolation and dLRP can be coupled in a straightforward way so that  $e_n^*\leq  e_n^*(\text{dLRP})$ almost surely.

In order to deduce further properties on the behavior of the undirected edges, see Corollary \ref{th:MaxEdgeLRPd1} below, we also consider the following variant of directed models.

\medskip
\noindent\underline{Long-Range Percolation in a quadrant (dLRPq):} Let us define the quadrant $\QQd$ as
\[\QQd:=\left\{x=(x_1,\dots, x_d)\in\mathbb{R}^d:\, x_1,\dots, x_{d-1}\geq 0, x_d > 0 \right\}.\] We call \emph{directed long-range percolation in the quadrant model} the random graph with vertex set $V$ in which an oriented edge $(u,v)$ is independently drawn from $u\in V$ to $v\in u+ \QQd$ with probability $g(v-u)$, $u,v \in V$. We write $x \underset{q}{\rightarrow} y$ if there is an oriented edge from $x$ to $y$ in dLRPq. Once again, we define the maximum edge-length
\begin{align*}
 e_n^*(\text{dLRPq}) &:=\max_{x\in V\cap B_n} \max_{y\in V\cap (x+\QQd)} \|x-y\| \mathbf{1}_{\{x \underset{q}{\rightarrow} y\}}.
\end{align*}
and the number of exceedances $W^{\underset{q}{\rightarrow}} (n,r)$, $n\in \N$, as in \eqref{direxc}, where we replace the expressions $\mathbf{1}_{\{x \rightarrow y\}}$ by $\mathbf{1}_{\{x \underset{q}{\rightarrow} y\}}$.
Moreover, (discrete and continuous) dLRPq and long-range percolation can also be coupled so that  $e_n^*(\text{dLRPq})\leq e_n^*$ almost surely.

\begin{theorem}\label{prop:MaxEdgedLRPq} Let $V=V^{dis}$. 
\begin{enumerate}
\item\label{d22m} If the function $g$ is given by \textnormal{(F)} with parameter $\alpha \in(d,\infty)$, let ${c}_n= ((\alpha-d)2^{-d}d^{-1}v_d^{-1})^\frac{1}{d-\alpha}n^\frac{d}{\alpha-d}$ and $b_n = 0$, $n\in \N$, where $v_d$ is a positive constant only depending on $d$. Then, for every $r\geq 0$, \eqref{eq:cvWn} when $W(n,r)$ is replaced by $W^\rightarrow(n,r)$ holds with $\mathbb{E}[W(r)] = r^{d-\alpha}$. In particular,
$$ c_n^{-1} e_n^*(\text{dLRP})  \xrightarrow[n\to \infty]{d} \Phi_{\alpha-d},$$
where $\Phi_{\alpha-d}$ has a Fr\'echet distribution with parameter $\alpha-d$.
\item\label{dirthg2}
  If the function $g$ is given by \textnormal{(G)} with parameters $\lambda,\alpha \in (0, \infty)$ define $K=K_{d, \alpha}= d2^d v_d\Gamma(\frac{d}{\alpha})\alpha^{-1}\lambda^{-\frac{d}{\alpha}}$. Further define 
  \begin{equation} 
      b_n=\lambda^{-\frac{1}{\alpha}}\left(d\ln n+\left(\frac{d}{\alpha}-1\right)\ln\left(d\ln n +\ln K\right)+\ln\left(K\Gamma\left(\frac{d}{\alpha}\right)^{-1}\right)\right)^\frac{1}{\alpha}
  \end{equation}
  and $c_n =\alpha^{-1}\lambda^{-1}b_n^{1-\alpha}$ for every $n\in\N$. Then, for every $r\in\rr$, \eqref{eq:cvWn} when $W(n,r)$ is replaced by $W^\rightarrow(n,r)$ holds with $\mathbb{E}[W(r)] = e^{-r}$. In particular,
$$ c_n^{-1} (e_n^*(\text{dLRP}) - b_n)  \xrightarrow[n\to \infty]{d} \Lambda,$$
where $\Lambda$ has a (standard) Gumbel distribution.
\end{enumerate}
Moreover, the same assertions hold with $W^\rightarrow(n,r)$ replaced by $W^{\underset{q}{\rightarrow}} (n,r)$, $e_n^*(\text{dLRP})$ replaced by $e_n^*(\text{dLRPq})$ and $v_d$ in part \eqref{d22m} and \eqref{dirthg2} is replaced by another suitably chosen constant.
\end{theorem}

Next, we state the corresponding results of Theorem \ref{prop:MaxEdgedLRPq} in the continuous model, which additionally includes the Weibull case (W).

\begin{theorem}\label{prop:MaxEdgedLRPq2} Let $V=V^{cont}$. Then the statements in Theorem \ref{prop:MaxEdgedLRPq} hold as well. In addition, if
the function $g$ is given by \textnormal{(W)} with parameters $M,\alpha \in (0, \infty)$, let ${c}_n = Kn^{-\frac{d}{\alpha+1}}$,  where $K=K_{ d, M,\alpha}=(M^\alpha(\alpha+1)2^{-d}d^{-1} v_d^{-1})^\frac{1}{\alpha+1}$, and let $b_n =M$, for every $n\in\N$. Then, for every $r\leq 0$, \eqref{eq:cvWn} when $W(n,r)$ is replaced by $W^\rightarrow(n,r)$ holds with $\mathbb{E}[W(r)] = (-r)^{\alpha+1}$. In particular,
$$ c_n^{-1} (e_n^*(\text{dLRP}) - M)  \xrightarrow[n\to \infty]{d} \Psi_{\alpha+1},$$
where $\Psi_{\alpha+1}$ has a Weibull distribution with parameter $\alpha+1$. Moreover, the assertion also holds  with $W^\rightarrow(n,r)$ replaced by $W^{\underset{q}{\rightarrow}} (n,r)$, $e_n^*(\text{dLRP})$ replaced by $e_n^*(\text{dLRPq})$ and $v_d$ replaced by another suitably chosen constant.
\end{theorem}

{A short summary of the contents is provided in Table \ref{tab:cont} for the reader's convenience. 
\begin{table}[ht]
    \begin{center}
\begin{tabular}{c|c|c|c}
Case & Norming constants & Value of $\mathbb{E}[W(r)]$ & Limit result\\
\hline
(F)& $c_n =K n^\frac{d}{\alpha-d}$, $b_n=0$  & $r^{d-\alpha}$ &   $c_n^{-1} e_n^*  \xrightarrow[n\to \infty]{d} \Phi_{\alpha-d}$\\
\hline
\multirow{2}{*}{(G)}& $c_n =\alpha^{-1}\lambda^{-1}b_n^{1-\alpha}$, & \multirow{2}{*}{$e^{-r}$} &   \multirow{2}{*}{$ c_n^{-1} (e_n^* - b_n)  \xrightarrow[n\to \infty]{d} \Lambda$}\\
&$b_n$: see \eqref{gumbelbn}  &&\\
\hline
(W), $V=V^{cont}$& $c_n =K  n^{-\frac{d}{\alpha+1}}$, $b_n=M$  & $(-r)^{\alpha+1}$ &   $c_n^{-1} (e_n^* - M)  \xrightarrow[n\to \infty]{d} \Psi_{\alpha+1}$\\
\end{tabular}
\end{center} 
\caption{Summary of the results for directed edges.}
\label{tab:cont}
\end{table}
}

Now Corollary \ref{th:MaxEdgeLRPd1} is an immediate consequence of the couplings between the three models.

\begin{cor}\label{th:MaxEdgeLRPd1} Let $V=V^{\text{dis}}$ or $V=V^{\text{cont}}$. The following assertions hold.
\begin{enumerate}

\item If the function $g$ is given by \textnormal{(F)} with the parameter $\alpha\in(d,\infty)$ then there exist $0<\kappa_1,\kappa_1'<\infty$ such that for $c_n=n^\frac{d}{\alpha-d}$, $n\in\N$, for every $r\in (0, \infty)$ we have
\[\limsup_{n \to \infty} \mathbb{P}(e_n^* \leq c_nr) \leq  \exp \left( -\kappa_1 r^{d-\alpha} \right)\]
and
\[\liminf_{n \to \infty} \mathbb{P}(e_n^* \leq c_nr) \geq  \exp \left( -\kappa'_1 r^{d-\alpha} \right).\]
\item If the function $g$ is given by \textnormal{(G)} with parameters $\lambda,\alpha \in (0,\infty)$ then there exist $0<\kappa_2,\kappa_2'<\infty$ such that for every $r\in \rr$ we have
\[\limsup_{n \to \infty} \mathbb{P}(e_n^*\leq c_{n} r +b_{n}) \leq  \exp \left( -\kappa_2 e^{-r} \right) \]
and
\[\liminf_{n \to \infty} \mathbb{P}(e_n^* \leq c_{n} r +b_{n}) \geq \exp \left( -\kappa_2' e^{-r} \right) , \]  
where $$b_n=\lambda^{-\frac{1}{\alpha}}\left(d\ln n+\left(\frac{d}{\alpha}-1\right)\ln\left(d\ln n +\ln K\right)+\ln\left(K\Gamma\left(\frac{d}{\alpha}\right)^{-1}\right)\right)^\frac{1}{\alpha}$$
for $K=K_{ d, \alpha}= d2^d\omega_d\Gamma(\frac{d}{\alpha})\alpha^{-1}\lambda^{-\frac{d}{\alpha}}$ and $c_n =\alpha^{-1}\lambda^{-1}b_n^{1-\alpha}$ for every $n\in\N$.
\item Let specifically $V=V^{\text{cont}}$. If the function $g$ is given by \textnormal{(W)} with the parameters $M, \alpha\in(0,\infty)$ then there exist $0<\kappa_3,\kappa_3'<\infty$ such that for $c_n$ and $b_n$, $n\in\N$, as specified in Theorem \ref{prop:MaxEdgedLRPq2} and for every $r\in (-\infty,0]$ we have
\[\limsup_{n \to \infty} \mathbb{P}(e_n^* \leq c_nr + b_n) \leq  \exp \left( -\kappa_3 (-r)^{\alpha+1} \right)\]
and
\[\liminf_{n \to \infty} \mathbb{P}(e_n^* \leq c_nr + b_n) \geq  \exp \left( -\kappa'_3 (-r)^{\alpha+1} \right).\]
\end{enumerate}
\end{cor}

\begin{remark}\it \label{rwd}
In the discrete model (both for the directed and undirected versions) a Weibull convergence fails, due to the fact that the edge lengths are discrete random variables. Indeed, if one defines $c_n^{-1} (e_n^* - M)$ as in the statement of Theorem \ref{thcnew} \eqref{cthw} then one can see that
$$c_n^{-1} (e_n^* - M) \to -\infty$$
almost surely, as $n\to \infty$, which is an interesting difference between the discrete and continuous model.
\end{remark}

Here are further brief comments about Corollary \ref{th:MaxEdgeLRPd1}. We can at least ensure that, in arbitrary dimension, under the regime in which the probability of an edge being present has a polynomial decay the maximum of the normalized edge lengths behaves almost like a Fr\'echet distribution with parameter $\alpha-d$, whereas it behaves almost like a Gumbel distribution in the case of exponential decay and so on.

The proofs of Theorems \ref{thd1}, \ref{thd2} and \ref{prop:MaxEdgedLRPq} for the discrete model and of Theorems \ref{th:MaxEdgeLRPc}, \ref{thcnew} and \ref{prop:MaxEdgedLRPq2} for the continuous model are split into Sections 3 and 4, respectively. We start with the proof of Theorem \ref{thd1} in the following Section.

\section{Proofs for the discrete model} \label{proofsdiscrete}

Let us recall that, for two non-negative, integer-valued random variables $X$ and $Y$, the \emph{total variation distance} $\d_{\text{TV}}$ and the \emph{Wasserstein distance} $\d_{\text{W}}$ between $X$ and $Y$ are respectively defined by
\[\d_{\text{TV}}\left(X,Y\right)=\sup_{A\subset \mathbb{N}_0} \left\vert \mathbb{P}\left(X\in A\right)-\mathbb{P}\left(Y\in A\right)\right\vert\]
and
\[\d_{\text{W}}\left(X,Y\right)=\sup_{f\in \text{Lip}(1)} \left\vert \mathbb{E}\left[f(X)\right]-\mathbb{E}\left[f(Y)\right]\right\vert\]
where $\mathbb{N}_0=\mathbb{N}\cup\{0\}$ and $\text{Lip}(1)$ denotes the set of Lipschitz functions from $\mathbb{N}_0$ to $\mathbb{R}$ with Lipschitz constant at most 1.


Our strategy in the following is to prove the Poisson convergence by establishing an upper bound on the total variation distance between the random variable $W(n,r)$ and its limit $W(r)$ in \eqref{eq:cvWn}. Then, by rewriting the exceedances as independent random variables we can use the so-called law of small numbers, see \cite[Theorem 4.6]{ross}. The main technical difficulty is then shifted to the exact calculation of expected values, which becomes quite tedious even in dimension one and essentially intractable in larger dimensions. \newline


\noindent \textit{Proof of Theorem \ref{thd1}.} {We will use \cite[Theorem 4.6]{ross} as described above.} Let $x_1, \ldots, x_{(2n+1)^d}$ be a given enumeration of $[-n,n]^d$ in $\integer^d$ and  $I=\{1, \ldots, (2n+1)^d\}$. Recall the definition of $W(n,r)$ in \eqref{def:NbExceed}. A crucial observation is that in all of the considered cases, namely $d \geq 1$, $\alpha \in (d,2d)$ and $r \in (0, \infty)$ or $d \geq 1$, $\alpha =2d$ and $r \in [2d ((\alpha-d)2^{-d}d^{-1}\omega_d^{-1})^\frac{1}{\alpha-d}, \infty)$, we have
\begin{align*}
    W(n,r) & = \sum_{x\in V \cap B_n} \sum_{y\in V } \mathbf{1}_{\{ \|x-y\| \mathbf{1}_{\{x\leftrightarrow y\}} >c_n r \}} - \frac12 \sum_{x\in V \cap B_n} \sum_{y\in V \cap B_n} \mathbf{1}_{\{ \|x-y\| \mathbf{1}_{\{x\leftrightarrow y\}} >c_n r \}} \\
    & = \sum_{x\in V \cap B_n} \sum_{y\in V } \mathbf{1}_{\{ \|x-y\| \mathbf{1}_{\{x\leftrightarrow y\}} >c_n r \}} .
\end{align*}
Let us denote
$$W(n,r) = \sum_{i=1}^{(2n+1)^d} \sum_{y \in \integer^d} X_i^n(y)$$
with the random variables of exceedance
 $$X^n_i (y)= \mathbf{1}_{\{\|x_i-y\| \mathbf{1}_{\{x_i\leftrightarrow y\}} >c_n r\}}  .$$
Note that
$$ p_n :=  \sum_{y \in \integer^d}\mathbb{E}[X_i^n(y)]=\sum_{y \in \integer^d\setminus B_{c_nr}(0)}g(y)$$
is independent of $i$ for all $i\in I$.
 Recall that, since we are in case (F), the connection function $g$ satisfies
$ \lim_{\|z\| \to \infty} g(z)\|z\|^{\alpha}= 1$ and that ${c}_n= ((\alpha-d)2^{-d}d^{-1}\omega_d^{-1})^\frac{1}{d-\alpha}n^\frac{d}{\alpha-d}$. Using 
the fact that
$$\int_{l}^\infty f(z) dz \leq \sum_{j=l}^\infty f(j) \leq \int_{l-1}^\infty f(z) dz$$
for $l \in \integer$ and monotonically decreasing functions $f$ it is easy to see that
\begin{equation*}
 \lim_{n \to \infty} \mathbb{E} [W(n,r)] = r^{d-\alpha}.
\end{equation*}
Now, let $P(n,r)$ be a Poisson random variable with mean $\mathbb{E} [W(n,r)]$. Since the total variation distance between two Poisson random variables is bounded by the absolute value of the difference of their parameters, it suffices to check that
\[\d_{\text{TV}}(W(n,r),P(n,r))\longrightarrow 0,\quad n\to \infty,\]
in order to show that $W(n,r)$ converges to $W(r)$ in distribution. Since $X^n_i(y)$ are independent, by \cite[Theorem 4.6]{ross} we thus have
	\[\d_{\text{TV}}(W(n,r),P(n,r)) \leq \min \{1, \mathbb{E} [W(n,r)]\} p_n \longrightarrow 0,\quad n\to \infty. \]
Thus we obtain the desired convergence result. Moreover, we conclude the proof of Theorem \ref{thd1} \eqref{d11} as follows. Note that 
$$\P(c_n^{-1}e_n^\star \leq r)=\P(W(n,r)=0) \to \exp (-r^{d-\alpha}), \quad n \to \infty.$$
Thus, $c_n^{-1}e_n^\star$ converges in distribution to a Fr\'echet random variable with parameter $\alpha-d$.
$\hfill \qed$
\medskip

\noindent \textit{Proof of Theorem \ref{thd2}.}
We only give a proof for the Fr\'echet case (F). The proof for the Gumbel case (G) only exhibits minor changes as compared to the Fr\'echet case when $\alpha > 2$.

By Theorem \ref{thd1} it only remains to consider the cases {\it (i)} $\alpha =2$ and $r\in (0,\frac12)$ and {\it (ii)} $\alpha >2$ and $r\in (0, \infty)$. We will proceed as in the proof of Theorem \ref{thd1}, but with minor modifications of the random variables $X_i^n(y)$, which nevertheless result in some crucial technical changes. The modifications come from the fact that we want to maintain writing the number of exceedances as a sum of independent random variables. At first, let us notice that by shift-invariance of the model we can consider $[0,2n]$ instead of $[-n,n]$ as observation window. This will be done in order to simplify the notations in the following calculations. Furthermore, {we enumerate the points in $[0,2n]$ in increasing order $x_1 = 0<x_2<\dots<x_{2n+1}=2n$,} $I=\{1, \ldots, 2n+1\}$ and define the random variables of exceedance
$$X^n_i(y) = \mathbf{1}_{\{{y \in \integer \setminus \{x_1, \ldots, x_i\} }, \|x_i-y\| \mathbf{1}_{\{x_i\leftrightarrow y\}} >c_n r\}} .$$
Note that {the $X_i^n(y)$ are independent by default, but
$$ p_n(i) :=  \sum_{y \in \integer}\mathbb{E}[X_i^n(y)]$$
now depends on $i\in I$. As before, we can conclude that
\[\d_{\text{TV}}(W(n,r),W(r))\longrightarrow 0,\quad n\to \infty,\]
as $n \to \infty$, and it only remains to calculate the limit of $$\mathbb{E} [W(n,r)]=\sum_{1\leq i\leq 2n+1}p_n(i).$$

Let


\[S(i,n,r)=\sum_{y \in B^C_{c_nr} (x_i)\setminus\{x_1,\dots, x_{i-1}\}} \|x_i-y\|^{-\alpha}\]
and $B_R (x_i) := \{y \in \integer : \|x_i-y\| \leq R\}.$ Then $S(i,n,r)$ obviously tends to 0, as $n\to\infty$.
Let us write $f_n\sim g_n$ if $\frac{f_n}{g_n}\rightarrow 1$, as $n\to\infty$. Since actually $\max_{1\leq i\leq 2n+1}S(i,n,r)\rightarrow 0$, as $n\to\infty$, we have:
\begin{align}
\E [W(n,r)]&=\sum_{i=1}^{2n+1}p_n(i)\sim\sum_{i=1}^{2n+1} S(i,n,r)\nonumber\\
	& =\sum_{i=1}^{2n+1} \Big( \sum_{y \in B^C_{{c}_nr} (x_i)} \|x_i-y\|^{-\alpha} - \sum_{k=1}^{i-1} \|x_i-x_k\|^{-\alpha} \mathbf{1}_{[0, \|x_k-x_i\|)} ({c}_nr) \Big) \nonumber\\
	& = \sum_{i=1}^{\lfloor {c}_n r\rfloor+1} \sum_{y \in B^C_{{c}_nr} (x_i)} \|x_i-y\|^{-\alpha}\nonumber \\
&\quad + \sum_{i=\lfloor {c}_n r\rfloor+2}^{2n+1} \Bigg( \sum_{y \in B^C_{{c}_nr} (x_i)} \|x_i-y\|^{-\alpha}  - \sum_{k=1}^{i-1} (i-k)^{-\alpha} \mathbf{1}_{[0, i-k)} ({c}_nr) \Bigg) \nonumber\\
	& = 2\left(\lfloor {c}_n r\rfloor+1\right) \sum_{j=\lfloor {c}_n r\rfloor+1}^\infty j^{-\alpha} + \sum_{i=\lfloor {c}_n r\rfloor+2}^{2n+1} \Bigg( 2\sum_{j=\lfloor {c}_n r\rfloor+1}^\infty j^{-\alpha}  - \sum_{k=1}^{i-\lfloor {c}_n r\rfloor-1} (i-k)^{-\alpha}  \Bigg)
	\nonumber\\
	& = 2\left(\lfloor {c}_n r\rfloor+1\right) \sum_{j=\lfloor {c}_n r\rfloor+1}^\infty j^{-\alpha} + \sum_{i=\lfloor {c}_n r\rfloor+2}^{2n+1} \Bigg( 2\sum_{j=\lfloor {c}_n r\rfloor+1}^\infty j^{-\alpha}  - \sum_{j=\lfloor {c}_n r\rfloor+1}^{i-1} j^{-\alpha}  \Bigg) \nonumber\\
	& = 2\left(\lfloor {c}_n r\rfloor+1\right) \sum_{j=\lfloor {c}_n r\rfloor+1}^\infty j^{-\alpha} + \sum_{i=\lfloor {c}_n r\rfloor+2}^{2n+1} \Bigg( \sum_{j=i}^\infty j^{-\alpha}    + \sum_{j=\lfloor {c}_n r\rfloor+1}^\infty j^{-\alpha}  \Bigg)\nonumber \\
	& = \left(2n+\lfloor {c}_n r\rfloor+2\right) \sum_{j=\lfloor {c}_n r\rfloor+1}^\infty j^{-\alpha} + \sum_{i=\lfloor {c}_n r\rfloor+2}^{2n+1}  \sum_{j=i}^\infty j^{-\alpha}\nonumber\\
	& \sim  \frac{1}{\alpha-1}\left(\left(2n+\lfloor {c}_n r\rfloor+2\right) \left(\lfloor {c}_n r\rfloor+1\right)^{1-\alpha} + \sum_{i=\lfloor {c}_n r\rfloor+2}^{2n+1}  i^{1-\alpha}\right) .\label{eq:dim1}
\end{align}
Recall that $c_n=Kn^{\frac{1}{1-\alpha}}$, where the value of the constant $K$ is given in Theorem \ref{thd1}. The result readily follows in the case $\alpha >2$ by definition of the constant $K$, since
$$\sum_{i=\lfloor c_n r\rfloor+2}^{2n+1}  i^{1-\alpha}$$
obviously forms a null sequence. 
Now, we focus on the case $\alpha =2$. Note that, with the notations of Lemma \ref{a1}, $K=K_{1,2}=4$ in this case. For $c_n=4n$ and $r\in (0,\frac12)$, the right hand side of \eqref{eq:dim1} converges to $\frac{1+2r}{2r} - \log (2r)$, as $n$ goes to $\infty$. This together with Theorem \ref{thd1} concludes the case $\alpha =2$. $\hfill \qed$
} \newline

The above computation shows that straightforward calculations become very tedious, even in dimension one, due to the corrective term appearing in {\eqref{eq:dim1}}. In fact, in dimension $d\geq 2$ the calculations essentially become intractable. The introduction of directed models gives a setting in which the random variables of exceedances are i.i.d.\,and thus the Poisson limit is easy to obtain (see Theorem \ref{prop:MaxEdgedLRPq}). Finally, the coupling between directed an undirected long-range percolation models permits to partially understand the asymptotic behavior of the longest edge in undirected LRP in the remaining cases (see Corollary \ref{th:MaxEdgeLRPd1}).

We now give the proof of Theorem \ref{prop:MaxEdgedLRPq} for dLRPq. The result for dLRP follows from the same lines and then Corollary \ref{th:MaxEdgeLRPd1} follows by using the coupling. 

\medskip

\noindent \textit{Proof of Theorem \ref{prop:MaxEdgedLRPq}.}
{We only give the proof of Theorem \ref{prop:MaxEdgedLRPq} for dLRPq. The result for dLRP follows from the same lines and then Corollay \ref{th:MaxEdgeLRPd1} follows by using the coupling. Since $W^{\underset{q}{\rightarrow}} (n,r)$ has a binomial distribution, it suffices to check that $\E[W^{\underset{q}{\rightarrow}} (n,r)]$ converges to the desired quantity, as $n\to\infty$.

Let $r\geq 0$ and assume first that $g$ is given by \textnormal{(F)} with some $\alpha\in(d,\infty)$. By arguing as in the proof of Theorem \ref{thd2}, one can see that
\begin{align*}
\E\left[W^{\underset{q}{\rightarrow}} (n,r)\right]&\sim \sum_{x_i\in B_n}\sum_{y \in B^C_{r_n} (x_i)\cap\left({x_i}+\QQd\right)} \|x_i-y\|^{-\alpha}\\&= (2n+1)^d\sum_{y \in B^C_{r_n}(0)\cap \QQd} \|y\|^{-\alpha}\\
&= (2n+1)^d\sum_{k=\lfloor r_n\rfloor +1}^\infty q_d(k) k^{-\alpha}
\end{align*}
where $q_d(k)=\#\{y\in\QQd:\,\Vert y\Vert=k\}=\frac{\kappa}{2^d} k^{d-1}+O(k^{d-2})$. Since $\alpha>d$, the latter quantity asymptotically behaves as
\begin{align*}
\frac{\kappa n^d}{\alpha-d} \lfloor Kc_nr\rfloor^{d-\alpha}+O\left(n^d\lfloor Kc_nr\rfloor^{d-\alpha-1}\right)= r^{d-\alpha}+O\left(n^\frac{d}{d-\alpha}r^{d-\alpha-1}\right)
\end{align*}
and the first assertion of Theorem \ref{prop:MaxEdgedLRPq} follows, by definition of the constant $K$ as in the proof of Theorem \ref{thd1} \eqref{d11}.
\medskip

Now we turn to the Gumbel case, let us now assume that $g$ is given by \text{(G)} for some $\lambda, \alpha \in (0, \infty)$. Similar computations as in the previous cases show that we are interested in the asymptotic behavior of
	\begin{align}
	 -\kappa n^d\sum_{k=r_n+1}^\infty  k^{d-1}\ln\left(1-\exp(-\lambda k^\alpha) \right)& \sim -\kappa n^d\int_{r_n}^\infty x^{d-1}\ln\left( 1-\exp(-\lambda x^\alpha)\right)\d x\nonumber\\
 & = c\kappa n^d\int_{r_n}^\infty x^{d-1}\exp(-\lambda x^\alpha)\d x \nonumber\\&= -\frac{\kappa}{\alpha} n^d\int_{r^\alpha_n}^\infty s^\frac{d-\alpha}{\alpha}\exp(-\lambda s)\d s\nonumber\\
	 &= c\frac{\kappa\Gamma\left(\frac{d}{\alpha}\right)}{\alpha\lambda^\frac{d}{\alpha}} n^d\P\left(Y>r^\alpha_n\right)\label{eq:asymGumbel2} 
	\end{align}}
\noindent for some suitable constant $c$, where  $Y$ is a Gamma-distributed random variable (see also the proof of Lemma \ref{a3} in the Appendix). As in the proof of Lemma \ref{a3} in the Appendix from the definition of the sequences $(c_n)_{n\in\N}$ and $(b_n)_{n\in\N}$ we get
	\begin{equation*}
		\lim_{n \to \infty} n \P(Y^{\frac{1}{\alpha}}> {c}_nr + {b}_n) = e^{-r}, \quad r\in \rr.
	\end{equation*}
	Together with \eqref{eq:asymGumbel2} this gives Theorem \ref{prop:MaxEdgedLRPq} by proceeding as above. Now the proof is complete.$\hfill \qed$

\medskip

\section{Proofs for the continuous model} \label{proofscont}

Theorems \ref{th:MaxEdgeLRPc} and \ref{thcnew} will be a consequence of Proposition \ref{pr:continuousCase} and the calculations in Appendix \ref{App}. The proof of Theorem \ref{prop:MaxEdgedLRPq2} is essentially a repetition of the arguments with only minor modifications. Recall again the definition of $W(n,r)$ from \eqref{def:NbExceed}. The following Proposition gives upper bounds for the total variation and the Wasserstein distances between the number $W(n,r)$ of exceedances in the observation window $B_n$ and a suitable Poisson random variable $P(n,r)$.

\begin{prop}\label{pr:continuousCase} Let $P(n,r)$ be a Poisson distributed random variable with mean $\beta_{n,r}=\mathbb{E}\left[W(n,r)\right]$. Then, 
\begin{equation*}d_{\text{TV}}\left(W(n,r),P(n,r)\right)\leq  2 (2n)^d \min(1,\beta_{n,r}^{-1})\left(\int_{B^C_{r_n}(0)}g(y)dy\right)^2
\end{equation*}
and 
\begin{equation*}d_{\text{W}}\left(W(n,r),P(n,r)\right)\leq 6 (2n)^d \min(1,\beta_{n,r}^{-\frac{1}{2}})\left(\int_{B^C_{r_n}(0)}g(y)dy\right)^2.
\end{equation*}
\end{prop}

In order to prove Proposition \ref{pr:continuousCase}, we will apply a variant of Theorem 3.1 in \cite{Penrose2}.  It deals with the case where the underlying marked Poisson point process $\eta$ consists of an homogeneous Poisson point process on $\mathbb{R}^d$ with intensity $\rho\in (0,\infty)$ and marks in some mark space $\left(\mathbb{M},\mathcal{M},\mathbf{m}\right)$ where $\mathbf{m}$ is a diffusive probability measure. {We will provide a construction of the model under consideration from a marked Poisson point process in Subsection \ref{ssec:construction}.} With a slight abuse of notation, we also denote by $\rho$ the measure defined by $\rho (\d x)=\rho\d x$. As noticed in that paper, only minor changes in the proof of \cite[Theorem 3.1]{Penrose2} lead to the version we will use below. We write $\mathbf{S}$ for the set of all locally finite subsets of $\mathbb{R}^d\times\mathbb{M}$ and $\mathbf{S}_k$ for the set of subsets of $\mathbb{R}^d\times\mathbb{M}$ of cardinality $k$, $k\in\mathbb{N}$.

\begin{theorem}[{\cite[Theorem 3.1.]{Penrose2}}] \label{Th:Penrose18}

Let $k\in\mathbb{N}$, $f:\,\mathbf{S}_k\times \mathbf{S} \longrightarrow \{0,1\}$ a measurable function and for $\xi\in \mathbf{S}$ set: 
\[F(\xi):=\sum_{\psi\in\mathbf{S}_k:\psi\subset\xi}f(\psi,\xi\setminus \psi).\]

Let $\eta$ be a (marked) Poisson point process with intensity $\rho\times \mathbf{m}$ in $\mathbb{R}^d\times \mathbb{M}$ and set $Z:=F(\eta)$ and $\beta:=\mathbb{E}[Z]$. For $x_1,\dots,x_k\in\mathbb{R}^d$, set $p(x_1,\dots,x_k):=\mathbb{E}\left[f\left((x_1,\tau_1),\dots,(x_k,\tau_k),\eta\right)\right]$ where the $\tau_i$ are independent random elements of $\mathbb{M}$ with common distribution $\mathbf{m}$. 

Suppose that for almost every $\mathbf{x}=(x_1,\dots,x_k)\in \left(\mathbb{R}^d\right)^k$ with $p(x_1,\dots,x_k)>0$ we can find coupled random variables $U_\mathbf{x}$ and $V_\mathbf{x}$ such that:
\begin{enumerate}
\item $Z\overset{d}{=}U_\mathbf{x}$,
\item $F\left(\eta\cup\overset{k}{\underset{i=1}{\bigcup}}\{(x_i,\tau_i)\}\right)$ conditional on $f\left(\overset{k}{\underset{i=1}{\bigcup}}\{(x_i,\tau_i)\},\eta\right)=1$ has the same distribution as $1+V_\mathbf{x}$,
\item $\mathbb{E}\left[\vert U_\mathbf{x}-V_\mathbf{x}\vert\right]\leq w(\mathbf{x})$ where $w$ is a measurable function.
\end{enumerate}

Let $P(\beta)$ be a mean $\beta$ Poisson random variable. Then

\begin{equation}d_{\text{TV}}\left(Z,P(\beta)\right)\leq\frac{\min(1,\beta^{-1})}{k!}\int_{(\mathbb{R}^d)^k}w(\mathbf{x})p(\mathbf{x})\operatorname{d} \mathbf{x}\label{th:Pdtv}\end{equation}
and 
\begin{equation}d_{\text{W}}\left(Z,P(\beta)\right)\leq\frac{3\min(1,\beta^{-\frac{1}{2}})}{k!}\int_{(\mathbb{R}^d)^k}w(\mathbf{x})p(\mathbf{x})\operatorname{d} \mathbf{x}.\label{th:PdW}\end{equation}

\end{theorem}  

Here are brief explanations regarding Theorem \ref{Th:Penrose18}. We will apply this Theorem with $k=2$ and a function $f$ that selects the points of $\eta \cap B_n$ where exceedances take place, so that $F$ is the number of exceedances.

\subsection{The typical maximal edge length}

In order to apply Theorem \ref{Th:Penrose18}, we will enlarge the random graph by adding two (marked) points at $x_1,x_2\in\mathbb{R}^d$ to define the coupled random variables $U_{x_1, x_2}$ and $V_{x_1, x_2}$. So, we will be led to work under the (2-fold) Palm measure associated with the underlying Poisson point process (see {\it e.g.$\,$} \cite[Sections 3.3 and 3.4]{SW} or \cite[Chapter 9]{last3} for an overview). We denote by $\mathbb{P}_{x_1, x_2}$ the Palm measure obtained by adding $x_1, x_2$ to the Poisson point process and by $\mathbb{E}_{x_1, x_2}$ the corresponding expectation.  

The next result is also of independent interest. To state it, let us recall that $e_0^*= \max_{y\in\p} \|y\| \mathbf{1}_{\{0 \leftrightarrow y\}}.$
Here, for the sake of generality, we work under the Palm measure $\mathbb{P}_0$ associated with a Poisson point process of intensity $\rho$ in $\mathbb{R}^d$. This also allows us to deduce the asymptotic behavior of $e_0^*$ when $\rho\to\infty$ (see Corollary \ref{cor:GrowingIntensities}).  

\begin{lemma} \label{tm}
	It holds
$$ \mathbb{P}_0( e_0^* \leq  r) = \exp \Big( -\rho   \int_{B_r^C(0)} g(x) dx \Big).$$
In particular, { for every $\varepsilon>0$, for $r$ large enough,} it holds
$$  (1-\varepsilon) \rho  \int_{B_r^C(0)} g(x) dx \leq \mathbb{P}_0( e_0^* > r) \leq \rho  \int_{B_r^C(0)} g(x) dx.$$
\end{lemma}

\begin{remark}\it
Lemma \ref{tm} gives the behavior of the tail distribution of $e_0^*$ with respect to the connection function $g$. In particular, one can see that $e_0^*$ has regularly varying tail in the Fr\'echet case \textnormal{(F)}, exponential tail in the Gumbel case \textnormal{(G)} and it has a power law behavior at its (finite) right endpoint in the Weibull case \textnormal{(W)}.
\end{remark}

\begin{proof}
	In this proof we borrow some ideas of the proof of \cite[Proposition 1]{Penrose1}. Let us note that the point process of those points that are connected by an edge to 0 in the random connection model constructed from $\mathcal{P}\cup\{0\}$ with connection function $g$ is a $g$ independent thinning of the Poisson point process $\mathcal{P}$. It is thus a Poisson point process with intensity $\rho g(\cdot)$ on $\mathbb{R}^d$ and the first claim follows. For the second claim recall that $1-\exp(-x) \sim x$ for $x \to 0$ and that $\rho \int_{B^C_r(0)} g(x) dx \to 0$, as $r \to \infty$ by assumption. Thus, for $r$ sufficiently large we get
	$$ (1-\varepsilon) \rho   \int_{B^C_r(0)} g(x) dx  \leq \mathbb{P}_0 ( E_{B^C_r(0)} \geq 1) \leq \rho   \int_{B^C_r(0)} g(x)dx ,$$
	and the corresponding claim is proven.
\end{proof}

From Lemma \ref{tm} and similar calculations as in the lemmas of the Appendix \ref{App} we get a following convergence result for increasing intensities, that is if $\rho \to \infty$, instead of growing observation balls, which is in the same spirit as some of the results in \cite{iyer}.
{\begin{cor}\label{cor:GrowingIntensities}
\begin{enumerate}
\item Assume that the function $g$ is given by \textnormal{(F)} with parameter $\alpha \in (d, \infty)$ and define $c_\rho= ((\alpha-d) \omega_d^{-1})^\frac{1}{d-\alpha}\rho^\frac{1}{\alpha-d}$. Then we have, under the Palm measure, that
$$c_\rho^{-1} e_0^* \xrightarrow[\rho \to \infty]{d} \Phi_{\alpha-d} ,$$
where $\Phi_{\alpha-d}$ has a Fr\'echet distribution with parameter $\alpha-d$.

\item Assume that the function $g$ is given by \textnormal{(G)} with parameters $\lambda,\alpha \in (0, \infty)$ and define $K=K_{ d, \alpha}= d2^d\omega_d\Gamma(\frac{d}{\alpha})\alpha^{-1}\lambda^{-\frac{d}{\alpha}}$. Further define 
  \begin{equation*}
      b_\rho=\lambda^{-\frac{1}{\alpha}}\left(\ln \rho +\left(\frac{d}{\alpha}-1\right)\ln\left(\ln \rho +\ln K\right)+\ln\left(K\Gamma\left(\frac{d}{\alpha}\right)^{-1}\right)\right)^\frac{1}{\alpha}
 \end{equation*} 
and $c_\rho =\alpha^{-1}\lambda^{-1}b_\rho^{1-\alpha}$. Then we have, under the Palm measure, that
$$ \lambda (e_0^* - b_\rho)  \xrightarrow[\rho\to \infty]{d} \Lambda,$$
where $\Lambda$ has a Gumbel distribution.
\item Assume that the function $g$ is given by \textnormal{(W)} with parameters $M,\alpha \in (0, \infty)$. Let $K=K_{ d, M,\alpha}=(M^\alpha(\alpha+1)2^{-d}d^{-1}\omega_d^{-1})^\frac{1}{\alpha+1}$ and define $c_\rho =K\rho^{-\frac{1}{\alpha+1}}$ and $b_\rho = M$. Then we have, under the Palm measure, that
$$ c_\rho^{-1} (e_0^* - M)  \xrightarrow[\rho\to \infty]{d} \Psi_{\alpha+1},$$
where $\Psi_{\alpha+1}$ has a Weibull distribution with parameter $\alpha+1$.
\end{enumerate}
\end{cor}}

\subsection{A construction of the random connection model}\label{ssec:construction}

We now recall how the random connection model can be constructed from a marked Poisson point process; see \cite[Section 4]{last2} for a similar construction. We choose $\mathbb{M}:=[0,1]^{\mathbb{N}\times\mathbb{N}}$ as mark space and $\mathbf{m}$ to be the distribution of a double sequence of independent random variables uniformly distributed on $[0, 1]$. Then we consider an independent $\mathbf{m}$-marking $\eta$ of $\mathcal{P}$ (that is a Poisson point process with intensity $\rho \times \mathbf{m}$ on $\mathbb{R}^d\times \mathbb{M}$; see {\it e.g.$\,$}\cite[Theorem 3.5.7]{SW}) and we fix a partition $\{D_i\}_{i\in\mathbb{N}}$ of $\mathbb{R}^d$ that consists of bounded Borel sets. For $x,x'\in\mathbb{R}^d$, we write $x'\lex x$ if $x'$ is smaller than $x$ in the lexicographic order. For $\left(x,\mathbf{u}=(u_{k,l})_{k,l\in\mathbb{N}}\right)$ and $i\in \mathbb{N}$, note that $\left\{x'\in\mathcal{P}\cap D_i:\, x'\lex x\right\}$ is a.s.$\,$finite since $D_i$ is bounded. {Thus, we can enumerate the elements of this set such that $x_1\lex x_2\lex \dots \lex x_r\lex x$ and set $U(\eta,x,x_j):=u_{i,j}$. Since $\mathcal{P}$ is a.s.$\,$simple, for any pair $\{x,x'\}$ of distinct points in $\mathcal{P}$, we have $x\lex x'$ or $x'\lex x$ thus $U(\eta,x,x')$ (if $x'\lex x$) or $U(\eta,x',x)$ (if $x\lex x'$) is well defined by the above procedure. If $U(\eta,x,y)$ is not defined by this procedure, we set $U(\eta,x,y)=1$.} Then, the random graph $G(\eta)$ with vertex set $\mathcal{P}$ and in which two distinct vertices $x' \lex x$ are connected by an edge if and only if $U(\eta,x,x')\leq g(x'-x)$ has the law of the random connection model with connection function $g$ in $\mathbb{R}^d$.

\subsection{Representation of the exceedances} \label{ssec:exrep}

For appropriate normalizing sequences $b_n,c_n$ and $r\in \rr$, we set $r_n=c_nr+b_n$. In the setting of Theorem \ref{th:MaxEdgeLRPc} we have
\begin{align*}
W(n,r) & = \sum_{x\in V \cap B_n} \sum_{y\in V } \mathbf{1}_{\{ \|x-y\| \mathbf{1}_{\{x\leftrightarrow y\}} >c_n r \}} - \frac12 \sum_{x\in V \cap B_n} \sum_{y\in V \cap B_n} \mathbf{1}_{\{ \|x-y\| \mathbf{1}_{\{x\leftrightarrow y\}} >c_n r \}} \\
    & = \sum_{x\in V \cap B_n} \sum_{y\in V } \mathbf{1}_{\{ \|x-y\| \mathbf{1}_{\{x\leftrightarrow y\}} >c_n r \}}.
\end{align*}
In this setting, for $k,l\in\mathbb{N}$, we denote by $A_{k,l}$ the set of $(x,\mathbf{u}, y,\mathbf{v},\xi)\in \mathbb{R}^d\times \mathbb{M}\times\mathbb{R}^d\times \mathbb{M}\times \mathbf{S}$ such that $x\in D_k\cap B_n$ and $y\in D_l$ satisfying
\begin{align*}
&\xi_{x,y}\left(\left\{(\tilde{y},\mathbf{\tilde{t}}):\, \Vert \tilde{y}-\tilde{x}\Vert\geq r_n,\tilde{y}\lex \tilde{x} \mbox{ and } U(\xi\cup\{(\tilde{x},\mathbf{\tilde{u}})\},\tilde{x},\tilde{y})\leq g(\tilde{x}-\tilde{y})\right\}\right)\\
&\quad +\xi_{x,y}\left(\left\{(\tilde{y},\mathbf{\tilde{t}}):\,\Vert \tilde{y}-\tilde{x}\Vert\geq r_n,\tilde{x}\lex \tilde{y} \mbox{ and } U(\xi\cup\{(\tilde{x},\mathbf{\tilde{u}})\},\tilde{y},\tilde{x})\leq g(\tilde{x}-\tilde{y})\right\}\right) ,
\\&\geq 1. 
\end{align*}
where $\xi_{x,y} := \xi\cup \{(x,\mathbf{u}, y,\mathbf{v})\}$. Then, we formally define $f:\, \mathbf{S}_2\times \mathbf{S}\longrightarrow \{0,1\}$ by
\begin{align*}
f\left(\{(x_1,\mathbf{u}_1), (x_2,\mathbf{u}_2)\},\xi\right)&=\mathbf{1}_{\{(x_1,\mathbf{u}_1,x_2,\mathbf{u}_2,\xi)\in \underset{k\geq 1}{\bigcup}\underset{l\geq 1}{\bigcup} A_{k,l}\}}
\end{align*}
for $(x_1,\mathbf{u}_1), (x_2,\mathbf{u}_2)\in \mathbb{R}^d\times \mathbb{M}$ and $\xi\in\mathbf{S}$. Note that $f$ is measurable by construction and that it is nothing but the indicator function of $\{\Vert y-x\Vert\mathbf{1}_{x\leftrightarrow y}\geq r_n,\,x\in B_n, y\in \rd\}$.

Now, in the setting of Theorem \ref{thcnew}, we have
\begin{align*}
W(n,r) & = \sum_{x\in V \cap B_n} \sum_{y\in V } \mathbf{1}_{\{ \|x-y\| \mathbf{1}_{\{x\leftrightarrow y\}} >c_n r \}} - \frac12 \sum_{x\in V \cap B_n} \sum_{y\in V \cap B_n} \mathbf{1}_{\{ \|x-y\| \mathbf{1}_{\{x\leftrightarrow y\}} >c_n r \}} ,
\end{align*}
which we rewrite as a sum of different indicators as follows, in order to ensure that exceedances are not counted twice:
\begin{align*}
W(n,r) & = \sum_{x\in V \cap B_n} \sum_{y\in V } \mathbf{1}_{\{ y \in \Tilde{B}_n(x) \cup B_n^C, \|x-y\| \mathbf{1}_{\{x\leftrightarrow y\}} >c_n r \}} ,
\end{align*}
where $\Tilde{B}_n(x):= \{ z \in B_n : x \leq z\}$. In this setting, for $k,l\in\mathbb{N}$, we denote by $A_{k,l}$ the set of $(x,\mathbf{u}, y,\mathbf{v},\xi)\in \mathbb{R}^d\times \mathbb{M}\times\mathbb{R}^d\times \mathbb{M}\times \mathbf{S}$ such that $x\in D_k\cap B_n$ and $y\in D_l \cap (\Tilde{B}_n(x) \cup B_n^C)$ satisfying
\begin{align*}
&\xi_{x,y}\left(\left\{(\tilde{y},\mathbf{\tilde{t}}):\, \Vert \tilde{y}-\tilde{x}\Vert\geq r_n,\tilde{y}\lex \tilde{x} \mbox{ and } U(\xi\cup\{(\tilde{x},\mathbf{\tilde{u}})\},\tilde{x},\tilde{y})\leq g(\tilde{x}-\tilde{y})\right\}\right)\\
&\quad +\xi_{x,y}\left(\left\{(\tilde{y},\mathbf{\tilde{t}}):\,\Vert \tilde{y}-\tilde{x}\Vert\geq r_n,\tilde{x}\lex \tilde{y} \mbox{ and } U(\xi\cup\{(\tilde{x},\mathbf{\tilde{u}})\},\tilde{y},\tilde{x})\leq g(\tilde{x}-\tilde{y})\right\}\right) ,
\\&\geq 1,
\end{align*}
where $\xi_{x,y} := \xi\cup \{(x,\mathbf{u}, y,\mathbf{v})\}$. Then, we formally define $f:\, \mathbf{S}_2\times \mathbf{S}\longrightarrow \{0,1\}$ by
\begin{align*}
f\left(\{(x_1,\mathbf{u}_1), (x_2,\mathbf{u}_2)\},\xi\right)&=\mathbf{1}_{\{(x_1,\mathbf{u}_1,x_2,\mathbf{u}_2,\xi)\in \underset{k\geq 1}{\bigcup}\underset{l\geq 1}{\bigcup} A_{k,l}\}}
\end{align*}
for $(x_1,\mathbf{u}_1), (x_2,\mathbf{u}_2)\in \mathbb{R}^d\times \mathbb{M}$ and $\xi\in\mathbf{S}$. Again, note that $f$ is measurable by construction and that it is nothing but the indicator function of $\{\Vert y-x\Vert\mathbf{1}_{x\leftrightarrow y}\geq r_n,\,x\in B_n, y\in \Tilde{B}_n(x) \cup B_n^C\}$.

We are now ready to apply Theorem \ref{Th:Penrose18} with
\begin{align*}
Z=W(n,r)&=F(\eta)=\sum_{(x,\mathbf{u}), (y,\mathbf{v})\in \eta}f\left(\{(x,\mathbf{u}), (y,\mathbf{v})\},\eta\setminus \{(x,\mathbf{u}),(y,\mathbf{v})\}\right)
\end{align*}
the number of exceedances when we select the vertices among the ones that belong to the observation window (without counting any exceedance twice). It remains to define the coupled random variables $U_{x_1,x_2}$ and $V_{x_1,x_2}$ and to control the upper bounds in \eqref{th:Pdtv} and \eqref{th:PdW}.

\subsection{Estimating $p(x_1,x_2)$}

We have $$p(x_1,x_2)=\mathbb{E}_{x_1,x_2}\Big[f\left(\{(x_1,\mathbf{u}_1), (x_2,\mathbf{u}_2)\}, \eta\right)\Big] ,$$ which is the probability that there is an exceedance at $x_1,x_2$ if we add $x_1,x_2$ equipped with independent random marks to the marked Poisson point process $\eta$. By definition
\begin{align}\label{eq:estp(x)}
\begin{split}
p(x_1,x_2)&=\mathbb{E}_{x_1,x_2} \left[\mathbf{1}_{\{x_1\in B_n,\, {x_2\in B^C_{r_n}(x_1)}, \,{x_1\leftrightarrow x_2}\}}\right] \\
&=\mathbf{1}_{\{x_1\in B_n\}} \mathbf{1}_{\{x_2\in B^C_{r_n}(x_1)\}}\mathbb{P}_{x_1,x_2}\left( x_1\leftrightarrow x_2 \right)\\
&=\mathbf{1}_{\{x_1\in B_n\}} \mathbf{1}_{\{x_2\in B^C_{r_n}(x_1)\}} g(x_1-x_2) .
\end{split}
\end{align}

\subsection{Defining $U_{x_1,x_2}$ and $V_{x_1,x_2}$}

Let us add two marked points $\mathbf{x}_1=(x_1, \tau_1)$, $\mathbf{x}_2=(x_2, \tau_2)$ to $\eta$ and consider the associated random connection model $G(\eta \cup \{\mathbf{x}_1, \mathbf{x}_2\} )$. Then we define the subgraph $G(\eta \cup \{\mathbf{x}_1, \mathbf{x}_2\} )_{|\eta}$ induced by the Poisson points in $\eta$ and we observe it has the same distribution as the original random connection model. We define $U_{x_1,x_2}$ as the number of exceedances in the induced graph $G(\eta \cup \{\mathbf{x}_1, \mathbf{x}_2\} )_{|\eta}$:
$$ U_{x_1,x_2}= \sum_{y\in \mathcal{P} \cap B_n} \sum_{z\in \mathcal{P} } \mathbf{1}_{\{ \|z-y\| \mathbf{1}_{\{z\leftrightarrow y\}} >r_n \}} - \frac12 \sum_{y\in \mathcal{P} \cap B_n} \sum_{z\in \mathcal{P} \cap B_n} \mathbf{1}_{\{ \|z-y\| \mathbf{1}_{\{z\leftrightarrow y\}} >r_n \}}.$$
 We also define $V_{x_1,x_2}$ as the number of exceedances in $B_n$ in the enlarged graph $G(\eta \cup \{\mathbf{x}_1, \mathbf{x}_2\} )$ other than the one at $x_1,x_2$ (if there is one), namely
\begin{align*}
 V_{x_1,x_2}&= \sum_{y\in (\mathcal{P} \cup \{x_1, x_2\}) \cap B_n} \sum_{z\in \mathcal{P} \cup \{x_1, x_2\} } \mathbf{1}_{\{ \|z-y\| \mathbf{1}_{\{z\leftrightarrow y\}} >r_n \}} \\
 & \quad - \frac12 \sum_{y\in (\mathcal{P} \cup \{x_1, x_2\}) \cap B_n} \sum_{z\in \mathcal{P} \cup \{x_1, x_2\} \cap B_n} \mathbf{1}_{\{ \|z-y\| \mathbf{1}_{\{z\leftrightarrow y\}} >r_n \}}\\
 & \quad - \mathbf{1}_{\{ \|x_1-x_2\| \mathbf{1}_{\{x_1\leftrightarrow x_2\}} >r_n \}} \Big( \mathbf{1}_{\{x_1 \in B_n\}} + \mathbf{1}_{\{x_2 \in B_n\}}  - \mathbf{1}_{\{x_1 \in B_n\}} \mathbf{1}_{\{x_2 \in B_n\}}\Big) .
\end{align*}

We note that $1+V_{x_1,x_2}$ has the desired conditional distribution and $V_{x_1,x_2}\geq U_{x_1,x_2}$ by construction.

\subsection{Estimating $w(x_1,x_2)$}

We have to provide an upper bound for
\begin{align*}
& w(x_1,x_2)=\mathbb{E}_{x_1,x_2}\left[\vert U_{x_1,x_2}-V_{x_1,x_2}\vert\right]=\mathbb{E}_{x_1,x_2}\left[ V_{x_1,x_2}-U_{x_1,x_2}\right]\\
&\leq \mathbb{E}_{x_1,x_2}\left[ \sum_{y \in B_n \cap \{x_1,x_2\} } \sum_{z \in  \mathcal{P}}  \mathbf{1}_{\{ \|z-y\| \mathbf{1}_{\{y\leftrightarrow z\}} > r_n\}} + \sum_{y \in \mathcal{P} \cap B_n } \sum_{z \in  \{x_1,x_2\}}  \mathbf{1}_{\{ \|z-y\| \mathbf{1}_{\{y\leftrightarrow z\}} > r_n\}}\right].
\end{align*}
It follows using the Mecke formula that

\begin{align} \label{eq:estw(x)} 
\begin{split}
w(x_1,x_2)&\leq \mathbf{1}_{\{x_1 \in B_n\}}\int_{\rd} \mathbb{P}_y\left(\|x_1-y\| \mathbf{1}_{\{y\leftrightarrow x_1\}} > r_n\right)d y \\
& \quad +  \mathbf{1}_{\{x_2 \in B_n\}}\int_{\rd} \mathbb{P}_y\left(\|x_2-y\| \mathbf{1}_{\{y\leftrightarrow x_2\}} > r_n\right)d y \\
&\quad + \int_{B_n} \mathbb{P}_y\left(\|x_1-y\| \mathbf{1}_{\{y\leftrightarrow x_1\}} > r_n\right)d y \\
& \quad + \int_{B_n} \mathbb{P}_y\left(\|x_2-y\| \mathbf{1}_{\{y\leftrightarrow x_2\}} > r_n\right)d y \\
& \leq 4 \int_{B^C_{r_n}(0)}g(y) dy.
\end{split}
\end{align}

\subsection{Concluding}

\noindent \textit{Proof of Proposition \ref{pr:continuousCase}.}
From \eqref{eq:estp(x)} and \eqref{eq:estw(x)}, we obtain the following upper bound for the integrals in the right-hand side of Equations \eqref{th:Pdtv} and \eqref{th:PdW}:

\begin{align*}
&\int_{\mathbb{R}^d} \int_{\mathbb{R}^d}w(x_1,x_2)p(x_1,x_2)dx_2 dx_1
= \int_{B_n} \int_{B^C_{r_n}(x_1)} w(x_1,x_2) g(x_1-x_2) dx_2dx_1 \\
&\leq 4\int_{B^C_{r_n}(0)}g(y)dy\times  \int_{B_n} \int_{B^C_{r_n}(x_1)}  g(x_1-x_2) dx_2dx_1\\
&\leq4 (2n)^d \left(\int_{B^C_{r_n}(0)}g(y)dy\right)^2.
\end{align*}
This along with Theorem \ref{Th:Penrose18} concludes the proof of Proposition \ref{pr:continuousCase}. $\hfill \qed$

\begin{proofof}{Theorem \ref{th:MaxEdgeLRPc}} We detail how to derive Theorem \ref{th:MaxEdgeLRPc} from Proposition \ref{pr:continuousCase} and Lemma \ref{a1}. As in the proof of Theorem \ref{thd1}, an important observation is that for the random variable $W(n,r)$ we have
\begin{align*}
    W(n,r) & = \sum_{x\in V \cap B_n} \sum_{y\in V } \mathbf{1}_{\{ \|x-y\| \mathbf{1}_{\{x\leftrightarrow y\}} >c_n r \}} - \frac12 \sum_{x\in V \cap B_n} \sum_{y\in V \cap B_n} \mathbf{1}_{\{ \|x-y\| \mathbf{1}_{\{x\leftrightarrow y\}} >c_n r \}} \\
    & = \sum_{x\in V \cap B_n} \sum_{y\in V } \mathbf{1}_{\{ \|x-y\| \mathbf{1}_{\{x\leftrightarrow y\}} >c_n r \}} 
\end{align*}
with
$$\mathbb{E} [W(n,r)] = \int_{B_n} \int_\rd g(x-y) dy dx= (2n)^d \int_{B^C_{r_n} (0)} g(z) dz. $$
Thus, by Lemma \ref{a1} the mean number of exceedances in $B_n$, $\beta_{n,r}=\mathbb{E}\left[W(n,r)\right]$, converges to $r^{d-\alpha}$, as $n$ tends to $\infty$. Let $P(n,r)$ and $W(r)$ be Poisson random variables with mean $\beta_{n,r}$ and $r^{d-\alpha}$ respectively. Again we note that, since the total variation distance between two Poisson random variables is bounded by the absolute value of the difference of their parameters, it suffices to check that
\[\d_{\text{TV}}(W(n,r),P(n,r))\longrightarrow 0,\quad n\to \infty,\]
in order to show that $W(n,r)$ converges to $W(r)$ in distribution. From Proposition \ref{pr:continuousCase}, we know that
\[\d_{\text{TV}}(W(n,r),P(n,r))\leq cn^d \left(\int_{B^C_{r_n} (0)} g(z) dz \right)^2\leq\frac{c'}{n^d}(n^d\beta_{n,r})^2,\]
with the notations of Lemma \ref{a1}. From that Lemma, we know the latter quantity vanishes. Hence, $W(n,r)$ converges to $W(r)$ in distribution. Finally, we note that 
$$\P(c_n^{-1}e_n^\star \leq r)=\P(W(n)=0) \to \exp (-r^{d-\alpha}), \quad n \to \infty.$$
Thus, $c_n^{-1}e_n^\star$ converges in distribution to a Fr\'echet random variable with parameter $\alpha-d$.
\end{proofof}

\begin{proofof}{Theorem \ref{thcnew}} 
We give a proof for the Fr\'echet case (F). The proof for the Gumbel case (G) and the Weibull case (W) only exhibits minor changes as compared to the Fr\'echet case when $\alpha > 2$. The proof of \eqref{eq:cvWn} follows exactly as in the proof of Theorem \ref{th:MaxEdgeLRPc}, and it only remains to calculate
$$ \lim_{n \to \infty} \mathbb{E} [W(n,r)],$$
which is more involved in the present case, i.e. the remaining cases {\it (i)} $\alpha =2$ and $r\in (0,\frac12)$ and {\it (ii)} $\alpha >2$ and $r\in (0, \infty)$. Recall that
\begin{align*}
W(n,r)&=  \sum_{x\in V \cap B_n} \sum_{y\in V } \mathbf{1}_{\{ \|x-y\| \mathbf{1}_{\{x\leftrightarrow y\}} >r_n \}} - \frac12 \sum_{x\in V \cap B_n} \sum_{y\in V \cap B_n} \mathbf{1}_{\{ \|x-y\| \mathbf{1}_{\{x\leftrightarrow y\}} >r_n \}},
\end{align*}
so that, using previous calculations and the Mecke formula
\begin{align*}
    \mathbb{E} [W(n,r)] & = 2n \int_{B^C_{r_n} (0)} g(z) dz - \frac12 \int_{B_n} \int_{B_n \cap B^C_{r_n}(x)} g(x-y) dy dx.
\end{align*}
Let us first consider the case $\alpha =2$ and $r\in (0,\frac12)$. Then it follows from Lemma \ref{a11new}
\begin{align*}
    \lim_{n \to \infty} \mathbb{E} [W(n,r)] & = \lim_{n \to \infty}  2n \int_{B^C_{r_n} (0)} g(z) dz - \frac12 \int_{B_n} \int_{B_n \cap B^C_{r_n}(x)} g(x-y) dy dx \\
    & = r^{-1} - \lim_{n \to \infty}\frac12 \int_{B_n} \int_{B_n \cap B^C_{r_n}(x)} \|x-y\|^{-2} dy dx \\
    & = r^{-1} - \frac12 \Big( \frac{1-2r}{r} + 2\log (2r)\Big)  \\
    & = \frac{1+2r}{2r} - \log (2r)
\end{align*}
as desired.

Now we turn to the case $\alpha>2$. Then it follows from Lemma \ref{a111new} and the definition of $r_n$
\begin{align*}
    \lim_{n \to \infty} \mathbb{E} [W(n,r)] & = \lim_{n \to \infty}  2n \int_{B^C_{r_n} (0)} g(z) dz - \frac12 \int_{B_n} \int_{B_n \cap B^C_{r_n}(x)} \|x-y\|^{-\alpha} dy dx \\
    & = 2r^{1-\alpha} - \frac12\lim_{n \to \infty} \int_{B_n} \int_{B_n \cap B^C_{r_n}(x)} g(x-y) dy dx \\
    & = 2r^{1-\alpha} - r^{1-\alpha}  \\
    & = r^{1-\alpha}
\end{align*}
as desired. This concludes the proof.
\end{proofof}

\appendix
\section{Technical lemmas} \label{App}
Given a connection function $g$ and the corresponding sequences $(c_n)_{n\in\N}$ and $(b_n)_{n\in\N}$ of norming constants, we define
$$ \beta_{n,r} = (2n)^d \int_{B^C_{r_n} (0)} g(z) dz $$
for every $n\in\N$, where $r_n = c_nr + b_n$. 

Recall that $\omega_d$ denotes the volume of the $d$-dimensional unit ball.

\begin{lemma}\label{a1}
Assume that $g$ is given by \textnormal{(F)} with parameter $\alpha \in (d, \infty)$. Let $K=K_{d, \alpha}=((\alpha-d)2^{-d}d^{-1}\omega_d^{-1})^\frac{1}{d-\alpha}$. Define $c_n =K n^{\frac{d}{\alpha-d}}$ and $b_n = 0$ for every $n\in\N$. Then, for every $r\geq 0$, it holds that
$$\lim_{n \to \infty} \beta_{n,r} =r^{d-\alpha}.$$
\end{lemma}
\begin{proof}
The assertion easily follows result for integration of spherical functions.
\end{proof}

\begin{lemma}\label{a11new}
Let $c_n=4n$ for every $n\in\N$. Then, for every $r \in (0,\frac12)$, it holds that
$$\lim_{n \to \infty} \int_{-n}^n \int_{[-n,n] \cap B^C_{c_n r}(x)} \|x-y\|^{-2} dy dx =\frac{1-2r}{r} + 2\log (2r).$$
\end{lemma}
\begin{proof}
    Direct calculations give
    \begin{align*}
	& \int_{-n}^n \int_{[-n,n] \cap B^C_{c_n r}(x)} \|x-y\|^{-2} dy dx = \int_{-n}^n \int_{B_n \setminus B_{c_nr}(x)} |y-x|^{-2} dy dx\\
	& = \int_{c_nr-n}^n \int_{c_n r}^{n+x} s^{-2} ds dx+ \int_{-n}^{n-c_nr} \int_{c_nr}^{n-x} s^{-2} ds dx\\
	& = \int_{c_nr-n}^n ((c_nr)^{-1}-(n+x)^{-1}) dx+ \int_{-n}^{n-c_n r} ((c_n r)^{-1}-(n-x)^{-1}) dx\\
	& =  \frac{1-2r}{2r}- \log(2n) + \log (c_n r) + \frac{1-2r}{2r} + \log (c_n r) -  \log(2n)\\
	&=\frac{1-2r}{r} + 2 \log (2r) ,
    \end{align*}
    by definition of $c_n$.
\end{proof}

\begin{lemma}\label{a111new}
Let $\alpha>2$ and define ${c}_n= ((\alpha-1)/2)^\frac{1}{1-\alpha}n^\frac{1}{\alpha-1}$ for every $n\in\N$. Then, for every $r \in (0,\infty)$ and $r_n:=c_n r$, it holds that
$$\lim_{n \to \infty} \int_{-n}^n \int_{[-n,n] \cap B^C_{r_n }(x)} \|x-y\|^{-\alpha} dy dx = 2r^{1-\alpha}.$$
\end{lemma}
\begin{proof}
    Define
    $$I(n):=  \int_{-n}^n \int_{[-n,n] \cap B^C_{r_n }(x)} \|x-y\|^{-\alpha} dy dx.$$
    Direct calculations give
    \begin{align*}
	I(n)=2\int_{0}^n \int_{B_n \setminus B_{r_n}(x)} \|y-x\|^{-\alpha} dy dx = 2I_1 (n) + 2I_2 (n)
	\end{align*}
	with
	\begin{align*}
	I_1(n) & = \int_{n-r_n}^n \int^{x-r_n}_{-n} \|y-x\|^{-\alpha} dy dx=\int_{n-r_n}^n \int^{n+x}_{r_n} s^{-\alpha} ds dx\\
	&=\frac{1}{1-\alpha}\int_{n-r_n}^n (n+x)^{1-\alpha}-{r_n}^{1-\alpha} dx\\
	& = \frac{1}{1-\alpha}  \frac{1}{2-\alpha} \left( (2n)^{2-\alpha} - \left(2n-r_n\right)^{2-\alpha} \right) - \frac{1}{1-\alpha} r_n^{2-\alpha}
	\end{align*}
	and
	\begin{align*}
	I_2(n) & = \int_{0}^{n-r_n} \left(\int^{x-r_n}_{-n} |y-x|^{-\alpha} dy  +  \int_{x+r_n}^{n} |y-x|^{-\alpha} dy\right) dx\\
	& =\int_{0}^{n-r_n} \left(\int^{n+x}_{r_n} s^{-\alpha} ds  +  \int_{r_n}^{n-x} s^{-\alpha} ds\right) dx\\
	& =\frac{1}{1-\alpha}\int_{0}^{n-r_n} \left(\left(n+x\right)^{1-\alpha}+\left(n-x\right)^{1-\alpha}-2r_n^{1-\alpha}\right) dx\\
	& =\frac{1}{1-\alpha}\frac{1}{2-\alpha}\left(\left(2n-r_n\right)^{2-\alpha}-r_n^{2-\alpha}\right)-\frac{2}{1-\alpha}(n-r_n)r_n^{1-\alpha}.
	\end{align*}
	It follows that, since $\alpha>2$,
	\begin{align*}
	I(n)&=\frac{2}{1-\alpha}\left(\frac{1}{2-\alpha}\left(\left(2n\right)^{2-\alpha}-r_n^{2-\alpha}\right)-(2n-r_n)r_n^{1-\alpha}\right)\\
	&\sim \frac{2}{1-\alpha} \Big((2n-r_n)r_n^{1-\alpha} \Big) \\
    & = \frac{r^{1-\alpha}}{n} (2n-r_n) \to 2r^{1-\alpha},
	\end{align*}
    as $n \to \infty$.
\end{proof}

\begin{lemma}\label{a3}
Assume that $g$ is given by \textnormal{(G)} with parameters $\lambda, \alpha \in (0, \infty)$. Let $K=K_{ d, \alpha}= d2^d\omega_d\Gamma(\frac{d}{\alpha})\alpha^{-1}\lambda^{-\frac{d}{\alpha}}$. Define $$b_n=\lambda^{-\frac{1}{\alpha}}\left(d\ln n+\left(\frac{d}{\alpha}-1\right)\ln\left(d\ln n +\ln K\right)+\ln\left(K\Gamma\left(\frac{d}{\alpha}\right)^{-1}\right)\right)^\frac{1}{\alpha}$$ and $c_n =\alpha^{-1}\lambda^{-1}b_n^{1-\alpha}$ for every $n\in\N$. Then, for every $r\in\mathbb{R}$, it holds that
$$\lim_{n \to \infty} \beta_{n,r} =e^{-r}.$$
\end{lemma}
\begin{proof}
By using a change to polar coordinates, one obtains that
\[\beta_{n,r}\sim (2n)^dd\omega_d\int_{r_n}^\infty \exp(-\lambda s^\alpha)s^{d-1}ds.\]
Then, by using the change of variables $t=s^\alpha$, one gets that
\begin{equation}\label{eq:bnG1}\beta_{n,r}\sim Kn^d\mathbb{P}\left( X^\frac{1}{\alpha}>r_n\right),\end{equation}
where $X$ is a random variable with Gamma distribution and parameters $\frac{\alpha}{d}$ and $\lambda$. 

For $x\in\mathbb{R}$, define $\overline{F}(x)=\mathbb{P}(X^\frac{1}{\alpha}>x)$, $\overline{G}(x)=\mathbb{P}(X>x)$ and $a(x)=\frac{x^{1-\alpha}}{\alpha\lambda}$. Observe that $(1+r\frac{a(x)}{x})^\alpha\underset{x\to \infty}{\sim} 1+\alpha r\frac{a(x)}{x}$. Thus, we have:
\begin{align}
 \lim_{x\to\infty}\frac{\overline{F}(x+ra(x))}{\overline{F}(x)}&=\lim_{x\to\infty}\frac{\overline{G}\left((x+ra(x))^\alpha\right)}{\overline{G}(x^\alpha)}=\lim_{x\to\infty}\frac{\overline{G}\left((x^\alpha+\alpha r\frac{a(x)}{x^{1-\alpha}}\right)}{\overline{G}(x^\alpha)}\nonumber\\&=\lim_{x\to\infty}\frac{\overline{G}(x+\frac{r}{\lambda})}{\overline{G}(x)}=e^{-r}.\label{eq:bnG2}  
\end{align}

Since $(b_n)_{n \in\N}$ is chosen such that $\lim_{n\to \infty}Kn^d\mathbb{P}(X>b_n^\alpha)=1$ (see \cite[Table 3.4.4]{EKM}), the claim follows from \eqref{eq:bnG1} and \eqref{eq:bnG2}.
\end{proof}

\begin{lemma}\label{a4}
Assume that $g$ is given by \textnormal{(W)} with parameters $M, \alpha \in (0, \infty)$. Let $K=K_{ d, M,\alpha}=(M^\alpha(\alpha+1)2^{-d}d^{-1}\omega_d^{-1})^\frac{1}{\alpha+1}$. Define $c_n =Kn^{-\frac{d}{\alpha+1}}$ and $b_n = M$ for every $n\in\N$. Then, for every $r\leq 0$, it holds that $\beta_{n,r} \equiv (-r)^{\alpha+1}$.
\end{lemma}
\begin{proof}
This assertion follows from direct calculations.
\end{proof}

\subsection*{Acknowledgement} ES acknowledges financial support from the Austrian Research Association. Large parts of this work were accomplished while ES was visiting the IMB at the University of Dijon. The IMB receives support from the EIPHI Graduate
School (contract ANR-17-EURE-0002).

\bibliographystyle{amsplain}
\bibliography{lit}

\end{document}